\documentclass[a4paper,12pt]{article}

\usepackage{amsmath}
\usepackage{amsfonts} 
\usepackage{amssymb,amsthm}
\usepackage{hyperref}
\usepackage{graphicx}
\usepackage[english]{babel}   
\usepackage{fullpage} 
\usepackage{tikz}
\usepackage{pb-diagram} 
\usepackage[cmtip,arrow]{xy}
\usepackage{pb-xy} 
\usepackage{enumerate}  
\usepackage[utf8]{inputenc} 
\usepackage{stmaryrd}

\usepackage{color}
\definecolor{red}{rgb}{1,.3,0.3}
\definecolor{gray}{gray}{0.5}
\definecolor{blue}{rgb}{0.2,0.5,1}

\newtheorem{de}{Definition}[section]
\newtheorem{theo}[de]{Theorem}
\newtheorem{prop}[de]{Proposition}
\newtheorem{cor}[de]{Corollary}
\newtheorem{lem}[de]{Lemma}
\newtheorem{rem}[de]{Remark}

\newcommand{\bb}[1][R]{\mathbb{#1}}
\newcommand{\sbb}[1][1|1]{\mathbb{R}^{#1}}

\newcommand{\cm}[1][M]{{\mathcal{#1}}} 
\newcommand{\cs}[1][S]{{\mathcal{#1}}} 
\newcommand{\cn}[1][N]{{\mathcal{#1}}}

\newcommand{\om}[1][M]{{\cm[O]_{\cm[#1]}}}

\newcommand{\tm}[1][M]{\cm[T]_{\cm[#1]}}

\newcommand{\pa}[1][q_i]{\partial_{#1}}
\newcommand{\sig}[1][|q_i|]{(-1)^{#1}}

\newcommand{\cmt}{\cm=(M,\om)}

\begin{document}

\renewcommand{\thefootnote}{\fnsymbol{footnote}} 

\title{A lossless reduction of geodesics on supermanifolds \\ to non-graded differential geometry}

\author{
{\bf St\'ephane Garnier\footnotemark[1] 
}
\\
{\bf Matthias Kalus\footnotemark[1] \footnotemark[2] 
}
\\
{\small  Fakult\"at f\"ur Mathematik}
\\
{\small Ruhr-Universit\"at Bochum}
\\ 
{\small D-44780 Bochum, Germany}
\\
}

\date{}
\footnotetext[1]{
partially supported by the SFB TR-12 of the Deutsche Forschungsgemeinschaft 
} 

\footnotetext[2]{
corresponding author, e-mail: matthias.kalus@rub.de
} 

\maketitle
\centerline{ {\bf Keywords: } Supermanifolds; Geodesics; Riemannian metrics; Connections}
\vskip0.3cm
\centerline{  {\bf MSC2010:} Primary 58A50, 53C22; Secondary  53B21, 53C05 }

\renewcommand{\thefootnote}{\arabic{footnote}}

\begin{abstract} \noindent
Let $\cm= (M,\mathcal O_\mathcal M)$ be a smooth supermanifold with connection $\nabla$ and Batchelor model $\mathcal O_\mathcal M\cong\Gamma_{\Lambda E^\ast}$. From  $(\cm,\nabla)$ we construct a connection on the total space of the vector bundle $E\to{M}$. This reduction of $\nabla$ is well-defined independently of the isomorphism $\mathcal O_\mathcal M \cong \Gamma_{\Lambda E^\ast}$. It erases information, but however it turns out that the natural identification of supercurves in $\cm$ (as maps from $ \mathbb R^{1|1}$ to $\mathcal M$) with curves in $E$ restricts to a 1 to 1 correspondence on geodesics. This bijection is induced by a natural identification of initial conditions for geodesics on $\cm$, resp.~$E$. 
Furthermore a Riemannian metric on $\mathcal M$ reduces to a symmetric bilinear form on the manifold $E$. Provided that the connection on $\cm$ is compatible with the metric, resp.~torsion free, the reduced connection on $E$ inherits these properties. For  an odd metric, the reduction of a Levi-Civita connection on $\cm$ turns out to be a Levi-Civita connection on $E$.
\end{abstract}

\section{Introduction}        

The analysis on supermanifolds of odd dimension one can be expressed in terms of the analysis on the associated Batchelor line bundle. Moreover morphisms from these objects to supermanifolds are determined by the induced pullback of numerical functions  and sections in the associated Batchelor bundles. 
The study of geodesics regarded as supercurves with domain $\sbb$ on a supermanifold $\cm$, can hence be completely described by classical differential geometry on  the Batchelor bundle of $\cm$. The aim of the present   article is to make this description precise.
\par \smallskip
Provided that $E \to M$ is a Batchelor bundle for the supermanifold $\cm=(M,\mathcal O_\mathcal M)$, there is a bijective correspondence between the set of supercurves  $\sbb \to \cm$ and the set of curves $\bb\to E$. We analyze this correspondence on the level of geodesics $\sbb\to \cm$ of a supermanifold with connection  $(\cm,\nabla)$.  In particular we construct a connection $\nabla^{TE}$ on the manifold $E$ such that the above correspondence of supercurves and curves yields a bijection between geodesics of $(\mathcal M,\nabla)$ and geodesics of $(E,\nabla^{TE})$. 
Assuming that $\nabla$ is torsion-free, resp.~metric with respect to an even or odd Riemannian metric $g$, we prove that $\nabla^{TE}$ is torsion-free, resp.~compatible with a constructed symmetric bilinear form $g^{TE}$ on the manifold $E$. In the case $|g|=1$ the form $g^{TE}$ is non-degenerate, so a Levi-Civita connection $\nabla$ on $(\mathcal M,g)$ induces a Levi-Civita connection $\nabla^{TE}$ on the pseudo-Riemannian manifold $(E,g^{TE})$.   \par \smallskip
Here is a detailed overview: In the second section we  state the well-known bijective correspondence of supercurves, resp.~initial conditions of geodesics on a supermanifold $\cm$ with curves, resp.~initial conditions on the classical manifold $E$, defining a Batchelor bundle for $\cm$  (see \cite{HKST}). Proofs are included for reasons of convenience of the reader. \par \smallskip
 The third section contains the explicit constructions of a connection and a symmetric bilinear form on $E$ induced by the respective objects on $\cm$. The main idea for the construction of $\nabla^{TE}$ and $g^{TE}$ is a natural identification of superfunctions of $\mathbb Z$-grading $0$ and $1$ in the Batchelor model with smooth functions on $E$ which are affine linear in fiber direction. Even (super-)derivations preserving these respective  sets of functions,  form Lie algebras. The identification of these Lie algebras enables us to transport a connection and an even or odd  Riemannian metric from the graded to the classical setting.  Furthermore we show that $\nabla^{TE}$ inherits zero torsion and metric compatibility if $\nabla$ has these properties. We show that the reduction commutes with diffeomorphisms of supermanifolds with connection. In particular it is independent of the choice of a Batchelor model. \par \smallskip
The fourth section then deduces the geodesics condition on $(E,\nabla^{TE})$ from the geodesics condition on $(\mathcal M,\nabla)$. We sum the results up to:
\begin{theo}\label{start} Let $\cm=(M,\mathcal O_\mathcal M)$ be a smooth supermanifold with Batchelor bundle $E\to M$, i.e.~$\mathcal O_\mathcal M\cong\Gamma_{\Lambda E^\ast}$, and connection $\nabla$, resp. Riemannian metric $g$. For the associated connection $\nabla^{TE}$, resp. symmetric bilinear form  $g^{TE}$ on $E$ we have:
\begin{itemize}
 \item[(1)] if $\nabla$ is torsion free, resp. compatible with $g$, then $\nabla^{TE}$ is torsion free, resp. compatible with $g^{TE}$,
 \item[(2)]  if $\Phi \in \text{Mor}(\sbb,\cm)$ is a  geodesic of $(\cm,\nabla)$ with initial condition $A$, then the associated curve $\gamma_\Phi \in \text{Mor}(\bb,E)$ is a geodesic of $(E,\nabla^{TE})$ with initial condition $\alpha_A$ associated to $A$.
\end{itemize}
\end{theo}
\textbf{Notation.} In the following we will denote a smooth finite-dimensional supermanifold $\mathcal M$ by the pair $(M,\mathcal O_\mathcal M)$ where $M=|\cm |$ is the underlying classical manifold and $\mathcal O_\mathcal M$ is the sheaf of superfunctions. We denote  by $\widetilde{f}$  the reduction of a superfunction $f$. Furthermore we will frequently identify a supermanifold $\mathcal M$ with the associated functor of points or generalized supermanifold
$$Y(\cm):\text{\textbf{SM}}^{\text{op}} \to \text{\textbf{Set}},\quad Y(\cm)(\mathcal S)=\text{\textbf{SM}}(\cs,\cm)$$
where $\text{\textbf{SM}}$ denotes the category of supermanifolds. We often write $\mathcal M(\mathcal S)$ for $\text{\textbf{SM}}(\mathcal S,\mathcal M)$ for given supermanifolds $\mathcal S$ and $\mathcal M$. Finally we recall the ``inner Hom'' functor associating to supermanifolds $\cm$ and $\cm[N]$  the generalized supermanifold $\underline{\text{\textbf{SM}}}(\cn,\cm):\text{\textbf{SM}}^{\text{op}}\to\text{\textbf{Set}}$ defined by
$\underline{\text{\textbf{SM}}}(\cn,\cm)(\cs):=\text{\textbf{SM}}(\cs\times \cn,\cm)$. \par \smallskip
By $\Pi$ we denote the {parity reversal} on the category of locally free sheaves of $\mathcal O_\mathcal M$-modules over $\cm$.  On classical vector bundles  we define $\Pi$ to be the Batchelor functor
 $$\text{\textbf{VB}}_{M}\to \text{\textbf{SM}}, \quad (E\to M) \mapsto \Pi E\ \widehat{=}\ (M,\Gamma_{\Lambda E^\ast}) \ .$$\par\smallskip
For certain calculations we use local coordinates of the bundle $\pi:E \to M$. Let $U \subset M$ be a coordinate domain for the coordinates $(x_i)$, let $(e_\alpha)$ be a basis of the vector space $F:=\mathbb R^{rank(E)}$ inducing coordinates $(e_\alpha^\ast)$ on $F$ and let $E|_U \cong U\times F$ be a local trivialization for the bundle $E \to M$ with coordinates $(x_i,e_\alpha^\ast)$.  The coordinates locally induce  derivations on $\mathcal C^\infty_E$ denoted by $\partial_i:=\partial_{x_i}$ and $\partial_\alpha:=\partial_{e_\alpha^\ast}$. Let  $\mathcal T_E$ denote the sheaf of derivations of $\mathcal C^\infty_E$.  
Furthermore the basis $(e_\alpha)$ induces by dualization a local frame also denoted $(e_\alpha^\ast)$ of the bundle $E^\ast \to M$ and hence local coordinates $(x_i,e_\alpha^\ast)$ for the supermanifold $\cm= \Pi E$. We obtain local superderivations on $\Gamma_{\Lambda E^\ast}$ by $\hat \partial_i:=\partial_{x_i}$ acting only on the coefficient functions of the products of the $e_\alpha^\ast$, and $\hat \partial_\alpha:=\partial_{e_\alpha^\ast}$ mapping $\mathcal C^\infty_M$ to zero.  \par \smallskip
\textbf{Acknowledgments.}  The authors would like to  thank Tilmann Wurzbacher for  discussions and advice during the completion of this article. 

\section{Supercurves and initial conditions}
We recall the natural identification of (super-)curves and initial conditions for geodesics in $\mathcal M=\Pi E$, resp. $E$. Since we will use the notion of the tangent bundle of a supermanifold as well as the notion of the sheaf of superderivations on superfunctions, we first remind the correspondence of locally free sheaves of $\mathcal O_\mathcal M$-modules of finite rank on the supermanifold $\mathcal M$ (category $\text{\textbf{LFS}}_{\cm}$) with  super vector bundles on $\mathcal M$ (category $\text{\textbf{SVB}}_{\cm}$). We quote from \cite{Schmitt}:
\begin{prop}\label{LFSvsSVB}\cite{Schmitt}  The categories $\text{\textbf{LFS}}_{\cm}$ and $\text{\textbf{SVB}}_{\cm}$ are equivalent. An equivalence is given by  
\[
 \begin{array}{rcl}
  \text{\textbf{LFS}}_{\cm} & \longrightarrow & \text{Fun}(\text{\textbf{SM}}^{\text{op}},\text{\textbf{Set}})\\
 \cm[P] & \longmapsto &  \sigma_\mathcal P:=\left(\cs  \longmapsto  \{(f,h)|f\in \cm(\cs)\text{ and }h\in(f^*\cm[P])_{\bar 0}(S)\} \right)                       
 \end{array}
\]
and the morphism of supermanifolds $\sigma_\mathcal P \to \mathcal M$ induced by the projection $(f,h) \mapsto f $.
\end{prop}
As usual we will denote the tangent, respectively cotangent bundles of a supermanifold $\mathcal M$ by $T\cm=(TM,\cm[O]_{T\cm})$, resp. $T^*\cm=(T^*M,\cm[O]_{T^*\cm})$ and the associated locally free sheaves by $\tm=Der_\mathbb R(\mathcal O_\mathcal M)$ and $\Omega_{\cm}^1=Hom_{\mathcal O_\mathcal M}(\mathcal T_\mathcal M,\mathcal O_\mathcal M)$ respectively. The above proposition allows to make sense of $\Pi$ in the category of super vector bundles.
\subsection{Identification of curves}\label{curve}
The main tool in comparing supercurves in $\mathcal M$ with curves in $E$ is the following theorem we quote (see, e.g.~\cite{HKST}):
\begin{theo}\label{SM} 
For a supermanifold $\mathcal M$ we have an isomorphism of generalized supermanifolds:
\begin{equation}\label{eq:91}
 \underline{\text{\textbf{SM}}}(\sbb[0|1],\cm)\cong \Pi T\cm
\end{equation} 
\end{theo}
We can further simplify the right side of equation (\ref{eq:91}) in the case of a fixed Batchelor model, i.e.~an isomorphism $\mathcal O_\mathcal M \to \Gamma_{\Lambda E^\ast}$:
\begin{prop}\label{E}
Let $E\to M$ be a real vector bundle over a classical real manifold $M$. Then we have 
$$|\Pi T\Pi E|\cong E\ . $$
\end{prop}
\begin{rem}\label{rem} The left hand side must be interpreted as follows: $T\Pi E  $ is the tangent
bundle $T(\Pi E )\to \Pi E $ of the supermanifold $\Pi E $ and $\Pi T\Pi E $ is the parity reversal of the super vector bundle $T(\Pi E )\to \Pi E $.
\end{rem}
\textit{Proof of Proposition \ref{E}.} Denote the locally free sheaf associated to $\Pi T\Pi E $ by $\mathcal P$. It is locally spanned by superderivations with reversed parity. In local coordinates $(x_i,e^\ast_\alpha)$ of $E$, the even part of $\mathcal P$ is spanned by $f_i\hat{\partial}_{i}$ and $f_\alpha\hat{\partial}_{\alpha}$ with $|f_i|=1$ and $|f_\alpha|=0$. The underlying manifold of $\Pi T\Pi E $ is given by the restriction of the functor $\sigma_\mathcal P$ in Proposition \ref{LFSvsSVB} to manifolds. For a manifold $S$ the set $\sigma_\mathcal P(S)$ consists of pairs $(f,h)$ for which $f$ can be identified with a morphism in $M(S)$ and the pullback $(f^*\cm[P])_{\bar 0}$ cancels the $f_i\hat{\partial}_{i}$ and reduces the coefficient functions in the $f_\alpha\hat{\partial}_{\alpha}$. So via dualization $h \in (f^*\cm[P])_{\bar 0}$ is a morphism $S \to E$ over $f:S \to M$ with respect to the projection $E\to M$. We have $\sigma_\mathcal P(S)\cong E(S)$ globally and naturality in $S$ can be easily shown as well.   \hfill   $\Box$ \par\bigskip

We obtain a 1:1 correspondence between the set of supercurves in $\cm$, i.e., morphisms
from $\sbb$ to $\cm$, and the set of curves in $E$:
\begin{cor}\label{T} Let $\cm$ be a supermanifold. If $E\to M$ is a real vector bundle such that $\cm$ is isomorphic to $\Pi E $ then $\cm(S \times \bb^{0|1})\cong E(S)$ for any classical manifold $S$. In particular 
$\cm(\sbb)\cong E(\bb)$ .
\end{cor}
\begin{proof} We have with Theorem \ref{SM}
$$ \cm(S \times \bb^{0|1}) = \text{\textbf{SM}}(S\times\sbb[0|1],\cm) = \underline{\text{\textbf{SM}}}(\sbb[0|1],\cm)(S) \cong  \Pi T\cm(S) \cong (\Pi T \Pi E  )(S)\ . $$
Since $S$ is classical, the right hand side equals $|\Pi T \Pi E  |(S)$ which is by Proposition \ref{E} isomorphic to $E(S)$. 
\end{proof}

\subsection{Identification of initial conditions} \label{ini}
The initial condition of a  geodesic is a morphism of supermanifolds from $\bb^{0|1}$ to $T \mathcal M$ (see \cite{SG-TW}). Hence we first analyze $T \cm$ for a fixed Batchelor model $\mathcal M=\Pi E$.
\begin{prop}\label{cit2}
Let $E\to M$ be a real vector bundle over a classical real manifold $M$. Then we have 
  $$|\Pi TT\Pi E |\cong TE \ . $$
\end{prop}
\begin{rem}
 The left hand side must be interpreted as follows: $T\Pi E  $ is the tangent
bundle $T(\Pi E )\to \Pi E $ of the supermanifold $\Pi E $ and $\Pi TT\Pi E $ is the parity reversal  of the super vector bundle $T(T\Pi E )\to T \Pi E $.
\end{rem}
\textit{Proof of Proposition \ref{cit2}.} 
We proceed analogously to the proof of Proposition \ref{E}. Denote the locally free sheaf associated to $\Pi TT\Pi  E $ by ${{\hat{\mathcal P}}}$. Local coordinates $(x_i,e^\ast_\alpha)$ on $E$ and the induced derivations on $\mathcal O_{\Pi E }$ induce coordinate functions $(x_i,e^\ast_\alpha,\hat\partial^\ast_j,\hat\partial^\ast_\beta)$ on $T\Pi E $ and associated super derivations $(\hat{\partial}_i,\hat{\partial}_\alpha,\partial_{\hat\partial_j},\partial_{\hat\partial_\beta})$. The even part of ${\hat{\mathcal P}}$ is spanned by $\hat{\partial}_{i}$ and $\partial_{\hat{\partial}_j}$ with odd coefficient superfunctions in $\mathcal O_{T\Pi E }$ and $\hat{\partial}_\alpha$ and $\partial_{\hat{\partial}_\beta}$ with even coefficient functions. Again restricting the associated functor $\sigma_{\hat{\mathcal P}}$ in Proposition \ref{LFSvsSVB} to manifolds, the set $\sigma_{\hat{\mathcal P}}(S)$ consists of pairs $(f,h)$ for which $f$ can be identified with a morphism in $TM(S)$ and $h \in (f^*\cm[P])_{\bar 0}$. Via $f^\ast$ we loose the $\hat{\partial}_{i}$ and ${\partial}_{\hat{\partial}_j}$ components and reduce the coefficient functions of the $\hat{\partial}_\alpha$ and $\partial_{\hat{\partial}_\beta}$ components to functions in $\mathcal O_{TM}$.  Recall the following  classical result (see, e.g~\cite{Michor}):\vspace*{0,2cm} \\ 
\textit{Let $E\overset{p}{\longrightarrow}M$ be a real vector bundle over a classical manifold $M$ with typical fiber $F$.
Then $TE\overset{Tp}{\longrightarrow}TM$ is a vector bundle with typical fiber $TF\cong F\oplus F$.}\vspace*{0,2cm} \\
Hence we obtain for $g$ the gluing properties for maps in $TE(S)$ over $f \in TM(S)$. Globally we obtain $\sigma_{\hat{\mathcal P}}(S)\cong TE(S)$ and the identification is  natural in $S$.
\hfill $\Box$ \par \bigskip
This leads to the identification of the later initial conditions for  geodesics on $\mathcal M$, resp. geodesics on $E$:
\begin{cor}\label{cor:xy}
 Let $\cm$ be a supermanifold. If $E\to M$ is a real vector bundle such that $\cm$ is isomorphic to $\Pi E $ then $T\cm(S\times\sbb[0|1])\cong TE(S)$ for any classical manifold $S$. In particular $T\cm(\sbb[0|1]) \cong TE$.
\end{cor}
\begin{proof} Let $S$ be a classical manifold, then using Theorem \ref{SM} 
\begin{eqnarray*}
 T\cm(S\times\sbb[0|1]) & = & \text{\textbf{SM}}(S\times\sbb[0|1],T\cm) =  \underline{\text{\textbf{SM}}}(\sbb[0|1],T\cm)(S) \cong  (\Pi TT\cm)(S)
\end{eqnarray*}
Since $S$ is classical, the right hand side equals $|\Pi TT \Pi E  |(S)$ which is by Proposition \ref{cit2} isomorphic to $TE(S)$. 
\end{proof}

\section{Connections and Riemannian metrics}
We first recall the  basic definitions of connections and metrics on supermanifolds and later construct the associated objects on the total space of the Batchelor bundle. Let $\cmt$ be a supermanifold and $\mathcal{P}$ be a locally free sheaf of $\om$-modules of finite rank on $\cm$.
\begin{de}
A \textbf{connection} on $\mathcal{P}$ is an even morphism $\nabla:\mathcal{P}\to
\Omega^1_{\cm}\otimes_{\om}\mathcal{P}$ of sheaves of real super vector spaces that satisfies the Leibniz rule
\[
 \nabla(f\cdot v)=df\otimes v+f\cdot \nabla (v)
\]
for all  $f\in\om(U)$, $v\in\mathcal{P}(U)$ and all open subsets $U\subset M$. Here $df(X):=\sig[|X||f|]X(f)$ for homogeneous $X\in\tm(U)$ and $f\in\om(U)$. Often we will regard $\nabla$ as a map $\nabla:\mathcal{T}_{\cm}\otimes_\mathbb R \mathcal{P} \to
\mathcal{P}$ that is $\mathcal O_\mathcal M$-linear in the first argument. In the case $\mathcal{P} = {\mathcal T}_\mathcal M$, we call $\nabla$ a {connection on} $\cm$.
\end{de}
 There is a natural graded analogue of \textbf{Christoffel symbols}: Let $(q_s )$ be a system of homogeneous 
coordinates for $\mathcal M$ on an open subset $U\subset M$. The expansion
\[
 \nabla_{\partial_s}\partial_u=\sum_r\Gamma_{su}^r\partial_r
\]
gives elements $\Gamma_{su}^r\in \cm[O]_\mathcal M (U )$ of parity $|\Gamma_{su}^r|=|q_s|+|q_u|+|q_r|$.
Furthermore the \textbf{torsion} of a connection $\nabla$ on $\cm$ is defined by
$$ T_\nabla (X, Y ):= \nabla_X Y-(-1)^{|X||Y |}\nabla_Y X - [X, Y ]$$
for $X,Y \in \tm(U)$.
\begin{de} An \textbf{even, resp. odd Riemannian metric} on a supermanifold $\cm$ is an 
even, resp. odd supersymmetric non-degenerate $\cm[O]_{\cm}$-linear morphism of sheaves
\[
g:\cm[T]_{\cm}\otimes_{\om}\cm[T]_{\cm}\to\cm[O]_{\cm}.
\]
Here non-degeneracy means that the mapping $g^{\natural}:X\mapsto g(X,.)$ is an isomorphism from $\cm[T]_{\cm}$ to
$\Omega^1_{\cm}$, resp. $\Pi\Omega^1_{\cm}.$ 
\end{de}
\begin{rem}\label{rem:b}
Note that an even Riemannian metric requires the odd dimension of $\mathcal M$ to be even, while an odd Riemannian metric requires that even and odd dimension of $\cm$ are equal. Moreover   $|g|=1$ for $M=\Pi E$   requires $E\cong TM$ as is shown in Appendix A. 
\end{rem}
\par \smallskip
Let $\cm$ be equipped with an even or odd Riemannian metric  $g$, then a connection $\nabla$ on $\mathcal M$ is called \textbf{metric} or compatible with $g$ if $\nabla g=0$, i.e.
\[
 \Theta_{\nabla,g}(X,Y,Z):=\sig[|g||X|] X(g(Y\otimes Z)) - g(\nabla_X Y\otimes Z) - (-1)^{|X||Y |}g( Y\otimes \nabla_X
Z) =0
\]
for all homogeneous $X,Y,Z\in\tm(U)$ and all open $U \subset M$. Parallel to the classical case it can be shown that on a supermanifold $\cm$ with an even or odd Riemannian metric $g$, there exists a unique torsion free and metric connection $\nabla$ (see \cite{SG-TW}). It is called the
\textbf{Levi-Civita connection} of the metric and its Christoffel symbols can be expressed explicitly in terms of the metric coefficients. 
\subsection{Construction of a connection and a bilinear form on $E$}\label{sec:con}
In the following we assume $\cm=\Pi E$ for a bundle $\pi:E \to M$, hence $\mathcal O_\mathcal M=\Gamma_{\Lambda E^\ast}$. Later in section \ref{sec} we analyze the stability of our constructions with respect to automorphisms of supermanifolds.
\par \smallskip
For the transport of objects from $\mathcal M$ to $E$ we give an identification of sufficiently large subsheaves $\mathcal F$, resp. $\mathcal F^\prime$ of the algebras of superfunctions on $\mathcal M$, resp. smooth functions on $E$. Moreover we identify the sheaves $\mathcal L_{\bar 0}$, resp. $\mathcal L_{\bar 0}^\prime$ of even (super-)derivations which preserve the subsheaves respectively.
Regard for any open set $U \subset M$
\begin{align*}
&\mathcal F(U):= \Gamma_{\mathbb R \oplus E^\ast}(U) \ \subset  \  \mathcal O_\mathcal M (U) \qquad \mbox{ and }\\
&\mathcal F^\prime(U):=\big\{f \in \mathcal C^\infty_E (\pi^{-1}(U)) \ \big| \ \exists w_f \in \mathcal C^\infty_M(U) \mbox{ such that } \\ & \qquad\qquad\qquad\qquad\qquad\qquad\qquad\qquad\forall \ x \in U:  \ f-\pi^\ast w_f \mbox{ is linear on } E_x \big\} \ .
\end{align*}
The sheaves $\mathcal F$ and $\mathcal F^\prime$ (via the pullback $\pi^\ast:\mathcal C^\infty_M(U) \to \mathcal C^\infty_E(\pi^{-1}(U))$) are sheaves of $\mathcal C^\infty_M$-modules on $M$. By the following map they are  isomorphic:
\begin{align*}
\psi:&\mathcal F(U) \to \mathcal F^\prime(U), \qquad s_{\bar 0}+s_{\bar 1} \mapsto \big(e \mapsto s_{\bar 0}(\pi(e))+s_{\bar 1}(\pi(e))(e)\big)\\
\psi^{-1}:&\mathcal F^\prime(U) \to \mathcal F(U), \qquad f \mapsto \big(x \mapsto w_f(x)+(f-\pi^\ast w_f)|_{E_x}\big)
\end{align*}
Set further:
\begin{align*}
&\mathcal L(U):=\big\{X \in \mathcal T_\mathcal M(U) \ \big| \ X|_{\mathcal F(U)} \subset \mathcal F(U) \big\} \qquad \mbox{ and }\\
&\mathcal L^\prime(U):=\big\{X^\prime \in \mathcal T_E(\pi^{-1}(U)) \ \big| \ X^\prime|_{\mathcal F^\prime(U)} \subset \mathcal F^\prime(U) \big\} \ .
\end{align*}
Conjugation with the isomorphism $\psi:\mathcal F \to \mathcal F^\prime$ yields an isomorphism of sheaves of $\mathcal C^\infty_M$-modules $\Psi:\mathcal L \to \mathcal L^\prime$. We use $\psi$ and $\Psi$ to transport the notion of even and odd elements from $\mathcal F=\mathcal F_{\bar 0}\oplus \mathcal F_{\bar 1}$, resp. $\mathcal L=\mathcal L_{\bar 0}\oplus \mathcal L_{\bar 1}$ to $\mathcal F^\prime$, resp. $\mathcal L^\prime$. The following lemma follows from the construction:
\begin{lem}
 $\Psi|_{\mathcal L_{\bar 0}}:\mathcal L_{\bar 0}\to\mathcal L_{\bar 0}^\prime$ is an isomorphism of sheaves of Lie algebras.
\end{lem}
The projection $pr: \mathcal O_\mathcal M= \mathcal F \oplus \Gamma_{\oplus_{k\geq 2}\Lambda^k E^\ast} \to \mathcal F$ yields a morphism of sheaves of $\mathcal C^\infty_M$-modules $p:\mathcal T_\mathcal M \to Hom_{\mathbb R}(\mathcal F,\mathcal F)$ by $X \mapsto pr\circ X|_{\mathcal F}$. Note that for all $f \in \mathcal C^\infty_M(U), \ s \in \mathcal F(U)$, $X \in \mathcal T_{\mathcal M, \bar 0}(U)$ and open $U\subset M$ it is due to reasons of the degree:
$$(pr\circ X|_\mathcal F)(f\cdot s)=pr\left(X(f)\cdot s+f \cdot X(s)\right)=(pr \circ X)(f)\cdot s+f\cdot (pr \circ X)(s)  $$
Hence $p(X)$ can be uniquely continued to $\mathcal O_\mathcal M$ by the graded Leibniz rule for the odd coordinates. So the  image of the even map $p$ lies in the sheaf $\mathcal L_{\bar 0}$ in a natural way.\par \smallskip 
We use the isomorphisms $\Psi|_{\mathcal L_{\bar 0}}$ and $\psi$ and the projections $p$ and $pr$ to transport items like a connection and later metric from $\cm$ to the target $E$. In both cases we will find that $\mathcal L_{\bar 0}^\prime\subset \mathcal T_E$ is a subspace of sufficient extent to determine the associated item on the target uniquely. We proceed in the following four steps:
\begin{enumerate}
 \item restrict the item on $\mathcal T_{\cm}$ by restriction and projection to an object on $\mathcal L_{\bar 0}$, resp. $\mathcal F_{\bar 0}$
 \item transport the item via the isomorphisms to $\mathcal L_{\bar 0}^\prime$, resp. $\mathcal F_{\bar 0}^\prime$
 \item use auxiliary functions $h \in \mathcal F^\prime$ to extend the item to $\mathcal L^\prime$, resp. $\mathcal F^\prime$
 \item use that $\mathcal L^\prime$ generates $\mathcal T_E$ as a $\mathcal C^\infty_E$-module to extend the item to $\mathcal T_E$
\end{enumerate}
The steps 3 and 4 need further notation. We introduce auxiliary functions for the third step in the following way: For homogeneous $R^\prime \in \mathcal L^\prime(U)$ with sufficiently small open $U \subset M$ there exists an auxiliary function  $h_{R^\prime}\in \mathcal F^\prime(U)$ of the same parity that is non-zero everywhere outside of the zero section and satisfies $h_{R^\prime}R^\prime\in \mathcal L_{\bar 0}^\prime(U)$. \par \smallskip
For the fourth step note that for $V \subset E$ open and $U:=\pi(V)$ we have the surjective multiplication map:
\begin{align}\label{eq:13}
m:\mathcal H_E(V) \to \mathcal T_E(V) \quad \mbox{ on } \quad \mathcal H_E(V):= \mathcal C^\infty_E(V) \otimes_\mathbb R \mathcal L^\prime(U)
\end{align}
Its kernel is generated (as a $\mathcal C^\infty_E$-module) by elements $1\otimes (f\cdot Z^\prime)-f\otimes Z^\prime$ with $Z^\prime \in \mathcal L_{\bar 0}^\prime(U)$ and $f \in \mathcal F_{\bar 0}^\prime(U)$ or $Z^\prime \in \mathcal L_{\bar 1}^\prime(U)$ and $f \in \mathcal F^\prime(U)$.\par \smallskip
We apply the outlined procedure first for the transport of a connection: 
\begin{prop}\label{nabla}
 A connection $\nabla$ on $\mathcal M$ induces a well-defined connection $\nabla^{TE}$ on $E$.
\end{prop}
\begin{proof}
\underline{Step 1:} The map $\check\nabla:=p \circ \nabla |_{\mathcal L_{\bar 0}\otimes_\mathbb R \mathcal L_{\bar 0}}:\mathcal L_{\bar 0}\otimes_\mathbb R \mathcal L_{\bar 0} \to \mathcal L_{\bar 0} $
which is a morphism of sheaves of $\mathbb R$-modules satisfies
\begin{align}\label{eq:01}
&\check\nabla_{fX}Y=f\check\nabla_XY  \\ \label{eq:02}
&\check\nabla_X(f Y)=X(f)Y+f \check\nabla_XY
\end{align}
 for all $U\subset M$ open, $ X,Y \in \mathcal L_{\bar 0}(U)$ and $f \in \mathcal C^\infty_M(U)$. \par
\underline{Step 2:} Via $\Psi$ we obtain $\nabla^{\prime}:\mathcal L_{\bar 0}^\prime\otimes_\mathbb R \mathcal L_{\bar 0}^\prime \to \mathcal L_{\bar 0}^\prime$ satisfying (\ref{eq:01}) and (\ref{eq:02}) for $ X,Y \in \mathcal L_{\bar 0}^\prime(U), \ f \in \pi^\ast(\mathcal C^\infty_M(U))$. Here we used $\Psi(fZ)=\psi(f)\Psi(Z)$ for $Z \in \mathcal L_{\bar 0}(U)$, $f \in \mathcal C^\infty_M(U)$.\par
\par
\underline{Step 3:} For homogeneous $X^\prime,Y^\prime \in \mathcal L^\prime(U)$ and auxiliary functions  $h_{X^\prime},h_{Y^\prime} \in \mathcal F^\prime(U)$ we have
\begin{align*}
  &\nabla_{h_{X^\prime}X^\prime}^{\prime} (h_{Y^\prime}Y^\prime) \stackrel{(\diamond)}{=} (\Psi \circ p)\Big( \psi^{-1}(h_{X^\prime}) \Psi^{-1}(X^\prime)\left(\psi^{-1}(h_{Y^\prime})\right)\Psi^{-1}(Y^\prime)\\ & \qquad\qquad\qquad\qquad\qquad+(-1)^{|X^\prime||Y^\prime|}\psi^{-1}(h_{X^\prime})\psi^{-1}(h_{Y^\prime}) \nabla_{\Psi^{-1}(X^\prime)}\Psi^{-1}(Y^\prime) \Big) \\ &\qquad\stackrel{(\ast)}{=}h_{X^\prime}X^\prime(h_{Y^\prime})Y^\prime+\left\{\begin{array}{ll}  0 &\ \mbox{for } |X^\prime|=|Y^\prime|=1 \\ h_{X^\prime} h_{Y^\prime}\cdot \left(\Psi \circ p \circ \nabla_{\Psi^{-1}(X^\prime)}\Psi^{-1}(Y^\prime)\right) &\ \mbox{else }  \end{array}\right.
\end{align*}
We have used for $(\diamond)$ that $\Psi^{-1}(h_{Y^\prime}Y^\prime)=\psi^{-1}(h_{Y^\prime})\Psi^{-1}(Y^\prime)$ and for $(\ast)$ additionally that $\nabla$ is an even map.
Hence we set in a unique and well-defined way
$$\nabla^{\prime}_{X^\prime}Y^\prime:=\frac{1}{h_{X^\prime}h_{Y^\prime}}\left(\nabla_{h_{X^\prime}X^\prime}^{\prime} (h_{Y^\prime}Y^\prime)-h_{X^\prime}X^\prime(h_{Y^\prime})Y^\prime\right) \ . $$
Note at this point that $\nabla^{\prime}$ extended to $\mathcal L^\prime \otimes_\mathbb R \mathcal L^\prime$ really maps to $\mathcal L^\prime$ still satisfying (\ref{eq:01}) and (\ref{eq:02}) for  $ X^\prime,Y^\prime \in \mathcal L^\prime(U), \ f \in \pi^\ast(\mathcal C^\infty_M(U))$.\par
\underline{Step 4:} With the notion in (\ref{eq:13}) we continue $\nabla^{\prime}$  to  a map $m\circ \nabla^{\prime}:\mathcal H_E\otimes_\mathbb R \mathcal H_E\to \mathcal T_E$ linearly in the first factor of the first and by Leibniz rule in the first factor of the second argument. By direct calculation this map is zero if any of the two arguments is in the kernel of $m$. (From Step 3 we see that only the case $1\otimes (f\cdot Z^\prime)-f\otimes Z^\prime$, $Z^\prime \in \mathcal L_{\bar 1}^\prime(U)$, $f \in \mathcal F_{\bar 1}^\prime(U)$ needs to be considered in detail -- here  $f$ is a suitable  auxiliary function $h_{Z^\prime}$.) Hence we finally obtain a well-defined map $\nabla^{TE}:\mathcal T_E \otimes_\mathbb R \mathcal T_E \to \mathcal T_E$ satisfying (\ref{eq:01}) and (\ref{eq:02}) for $X,Y\in \mathcal T_E(U)$ and $f \in\mathcal C^\infty_E(U)$.
\end{proof}
A transport analogue to Proposition \ref{nabla} is possible for a metric yielding a symmetric bilinear form:
\begin{prop}\label{g}
 An even or odd Riemannian metric $g$ on $\cm$ induces a well-defined symmetric bilinear form $g^{TE}$ on $E$.
\end{prop}
\begin{proof}
\underline{Step 1:} $g$ induces a map $\check g:=pr \circ g|_{\mathcal L_{\bar 0} \otimes \mathcal L_{\bar 0}}:\mathcal L_{\bar 0} \otimes \mathcal L_{\bar 0} \to \mathcal F$.\par
\underline{Step 2:} Via $\Psi$ and $\psi$ we obtain a map $g^{\prime}:\mathcal L_{\bar 0}^\prime \otimes \mathcal L_{\bar 0}^\prime \to \mathcal F^\prime$ which is $\mathcal C^\infty_M$-linear and symmetric. \par
\underline{Step 3:} For homogeneous $X^\prime,Y^\prime \in \mathcal L^\prime(U)$ and $h_{X^\prime},h_{Y^\prime}\in \mathcal F^\prime(U)$ auxiliary functions we have
\begin{align*}
 &g^{\prime}(h_{X^\prime}X^\prime,h_{Y^\prime}Y^\prime)&=(\psi \circ pr)\big((-1)^{|g|(|X^\prime|+|Y^\prime|)}\psi^{-1}(h_{X^\prime})\psi^{-1}(h_{Y^\prime})g(\Psi^{-1}(X^\prime),\Psi^{-1}(Y^\prime))\big)
\end{align*}
\begin{align}\label{eq:15}
 &=\left\{ \begin{array}{ll}h_{X^\prime}h_{Y^\prime} (\psi \circ pr)\big(g(\Psi^{-1}(X^\prime),\Psi^{-1}(Y^\prime))\big)& \mbox{ if } |g|=|X^\prime|=|Y^\prime|=0\\ h_{X^\prime}h_{Y^\prime}(-1)^{(|X^\prime|+|Y^\prime|)} (\psi \circ pr)\big(g(\Psi^{-1}(X^\prime),\Psi^{-1}(Y^\prime))\big)&\mbox{ if }  |g|\cdot (1-|X^\prime||Y^\prime|)=1\\ 0 &\mbox{ else} \end{array} \right.
\end{align}
 So we set
 $$g^{\prime}(X^\prime,Y^\prime):=\frac{1}{h_{X^\prime}h_{Y^\prime}}g^{\prime}(h_{X^\prime}X^\prime,h_{Y^\prime}Y^\prime)$$ 
\underline{Step 4:} 
At last we $\mathcal C^\infty_E$-linearly continue $g^{\prime}$ to $\mathcal H_E\otimes_{\mathcal C^\infty_E} \mathcal H_E$ via (\ref{eq:13}). If any argument of $g^{\prime}$ is in the kernel of $m$, the result will vanish. Hence this construction yields a $\mathcal C^\infty_E$-linear symmetric map $g^{TE}:\mathcal T_E\otimes_{\mathcal C^\infty_E} \mathcal T_E \to \mathcal C^\infty_E$. 
\end{proof}
\subsection{Properties of $\nabla^{TE}$ and $g^{TE}$}\label{coord}
We determine the Christoffel symbols of $\nabla^{TE}$ and metric coefficients of $g^{TE}$ in the notion of coordinates presented in the introduction and compare torsion and metric compatibility with the respective properties of $\nabla$ and $g$. 
We denote the Christoffel symbols of $\nabla$ with respect to $(\hat \partial_i,\hat \partial_\alpha)$ by $\Gamma^{s}_{u,r}\in \mathcal O_\mathcal M(U)$ and the metric coefficients of $g$ by $g_{r,s}$ with $s,u,r$ being any even or odd indexes.
\begin{lem}\label{chris}
 The Christoffel symbols of $\nabla^{TE}$ with respect to $( \partial_i,\partial_\alpha)$ are:
 \begin{align*}
  &{\Gamma^{TE}}_{ij}^k=  \widetilde{\Gamma}_{ij}^k \circ \pi, &&{\Gamma^{TE}}_{ij}^\alpha= \sum_\beta e_\beta^\ast \left(\widetilde {\hat\partial_\beta \Gamma_{ij}^\alpha} \circ \pi\right)\\
  & {\Gamma^{TE}}_{i\alpha}^\gamma=\widetilde{\Gamma}_{i\alpha}^\gamma \circ \pi, && {\Gamma^{TE}}_{\alpha i}^\gamma=\widetilde{\Gamma}_{\alpha i}^\gamma \circ \pi 
 \end{align*}
 $${\Gamma^{TE}}_{i\alpha}^k={\Gamma^{TE}}_{\alpha i}^k={\Gamma^{TE}}_{\alpha\beta}^k={\Gamma^{TE}}_{\alpha\beta}^\gamma=0  $$
\end{lem}
\begin{proof}
Comparing the results from applying the above definition of $\nabla^{TE}$ on the appropriate derivations in $\mathcal L_{\bar 0}^\prime$ with the results obtained by using linearity for the auxiliary functions in the lower and Leibniz rule in die upper argument, yields the Christoffel symbols.
\end{proof}
By direct calculation we have for the metric coefficients:
\begin{lem}\label{metric}
 The metric coefficients of $g^{TE}$ with respect to $( \partial_i,\partial_\alpha)$ are:
 \begin{align*}
& \mbox{for } |g|=0 \ :\qquad && g_{ij}^{TE}= \widetilde{g}_{ij} \circ \pi&&  g_{i\alpha}^{TE}=0&& g_{\alpha i}^{TE}=0&& g_{\alpha\beta}^{TE}=0\\
& \mbox{for } |g|=1 \ : && g_{ij}^{TE}= \sum_\beta e_\beta^\ast\left(\widetilde{\hat \partial_\beta g_{ij}}\circ \pi \right)&&  g_{i\alpha}^{TE}=-\widetilde{g}_{i\alpha}\circ \pi&&   g_{\alpha i}^{TE}=-\widetilde{g}_{\alpha i}\circ \pi&&   g_{\alpha\beta}^{TE}=0
\end{align*} 
\end{lem}
We conclude from the metric coefficients:
\begin{cor}
 In the case $|g|=1$ the symmetric bilinear form $g^{TE}$ is a pseudo-Riemannian metric on the manifold $E$. In the case $|g|=0$ the form $g^{TE}$ is degenerate as soon as $\cm$ has positive odd dimension.
\end{cor}
Comparing graded and non-graded objects we find furthermore:
\begin{lem}
 If $\nabla$ is torsion-free then $\nabla^{TE}$ is torsion-free.
\end{lem}
\begin{proof}
The torsion tensor is bilinear in functions, so it is sufficient to calculate $T_{\nabla^{TE}}(X,Y)$ for $X,Y \in \mathcal L_{\bar 0}^\prime$. It is $T_{\nabla^{TE}}(X,Y)=(\Psi\circ p)\left(T_\nabla(\Psi^{-1}(X),\Psi^{-1}(Y))\right)$ due to an analysis of the $\mathbb Z$-grading.
\end{proof}
\begin{lem}\label{metric2}
 If $\nabla$ is compatible with $g$ then $\nabla^{TE}$ is compatible with $g^{TE}$. 
\end{lem}
\begin{proof}
The identity $\Theta_{\nabla^{TE},g^{TE}}(X,Y,Z)=(\Psi\circ p)\left(\Theta_{\nabla,g}(\Psi^{-1}(X),\Psi^{-1}(Y),\Psi^{-1}(Z))\right)$ for $X,Y,Z \in \mathcal L_{\bar 0}^\prime$ holds with an analogue argument and suffices to prove the lemma.
\end{proof}
\begin{rem}\label{rem:a}
 The obvious correction to make $g^{TE}$ non-degenerate in the case  $|g|=0$ by requiring $g_{\alpha\beta}^{TE}=\widetilde{g}_{\alpha\beta} \circ \pi$ in Lemma \ref{metric}, in general destroys compatibility of the connection. In Appendix A a different reduction yielding exactly the information of $\widetilde{g}_{rs} \circ \pi$ for $r,s$ even or odd indexes, is given.
\end{rem}
The lemmas yield the first statement of Theorem \ref{start} and include further:
\begin{cor}\label{coro}
In the case $|g|=1$ we have: If $\nabla$ is Levi-Civita for $g$ then $\nabla^{TE}$ is Levi-Civita for $g^{TE}$.
\end{cor}
\begin{rem}\label{reM}
Note that the reduction map $M \to \mathcal M$ and the zero section $M \to E$ allow pullbacks of $\nabla$, resp. $\nabla^{TE}$ to $M$. Both coincide as connections $\nabla^{TM}$ on $M$. In \cite{SG-TW} it is shown that if $\nabla$ is Levi-Civita with respect to an even Riemannian metric $g$, then $\nabla^{TM}$ is Levi-Civita with respect to the reduction $g^{TM}$ of $g$ on $M$.
\end{rem}
As is evident from the Lemmas \ref{chris} and \ref{metric}, the reduction $\nabla \mapsto \nabla^{TE}$ and $g\mapsto g^{TE}$ looses much information. Hence there is no canonical reconstruction of $\nabla$ and $g$ from the classical objects  $\nabla^{TE}$ and $g^{TE}$.
\subsection{Dependence on diffeomorphisms}\label{sec}
So far we have identified $\mathcal M$ with $\Pi E$, but this identification is fixing one of possibly many isomorphisms $\mathcal O_\mathcal M\cong \Gamma_{\Lambda E^\ast}$. Here we analyze the dependence of our construction on this choice, so its behavior under diffeomorphisms. Let now $\mathcal M=(M,\mathcal O_\mathcal M)$, $\mathcal N=(N,\mathcal O_\mathcal N)$ be two supermanifolds  identified with their Batchelor models, so $\mathcal O_\mathcal M=\Gamma_{\Lambda E^\ast}$ and $\mathcal O_\mathcal N=\Gamma_{\Lambda F^\ast}$ for vector bundles $E\to M$ and $F \to N$. Let $\nabla^\mathcal M$, $\nabla^\mathcal N$ be connections on $\mathcal M$, respectively $\mathcal N$. We mark the sheaves and morphisms $\mathcal F,\ \mathcal L, \ \Psi, \ldots$ with upper index $\mathcal M$, respectively $\mathcal N$. Furthermore let $\Phi:\mathcal M\to \mathcal N$ be a diffeomorphism of supermanifolds with underlying map $\varphi:M\to N$. Then we obtain an  isomorphism of sheaves of $\mathcal O_\mathcal N$-modules on $N$ given by 
$\Phi_\ast:\varphi_\ast\mathcal T_\mathcal M \to \mathcal T_\mathcal N$ via conjugation with $\Phi^\ast$. \par\smallskip
The composition $p^\mathcal N\circ \Phi_\ast$ yields via the identification $\mathcal L_{\bar 1}^\mathcal M\cong \Gamma_E$ and $\mathcal L_{\bar 1}^\mathcal N\cong \Gamma_F$, a vector bundle isomorphism $\hat \Phi:=p^\mathcal N\circ \Phi_\ast|_{\mathcal L_{\bar 1}^\mathcal M}:E \to F$. The later satisfies  $\hat \Phi^\ast(\mathcal F^{\prime \mathcal N})=\varphi_\ast\mathcal F^{\prime \mathcal M}$ and hence induces $\hat\Phi_\ast|_{\varphi_\ast \mathcal L_{\bar 0}^{\prime\mathcal M}}:\varphi_\ast \mathcal L_{\bar 0}^{\prime\mathcal M} \to \mathcal  L_{\bar 0}^{\prime\mathcal N}$.\par \smallskip
We call $\Phi:(\mathcal M, \nabla^\mathcal M)\to(\mathcal N,\nabla^\mathcal N)$ a \textbf{diffeomorphism of supermanifolds with connection}, if it satisfies $\Phi_\ast \circ \nabla^\mathcal M=\nabla^\mathcal N \circ \otimes_\mathbb R^2 \Phi_\ast$. Analogously for the classical map $\hat \Phi:(E,\nabla^{TE})\to(F,\nabla^{TF})$.
 We can follow:
\begin{lem}
 If $\Phi:(\mathcal M, \nabla^\mathcal M)\to(\mathcal N,\nabla^\mathcal N)$ is a {diffeomorphism of supermanifolds with connection}, then $\hat \Phi:(E,\nabla^{TE})\to(F,\nabla^{TF})$ inherits this property.
\end{lem}
\begin{proof} 
By section \ref{sec:con} it is sufficient to prove $\hat\Phi_\ast \circ \nabla^{TE}=\nabla^{TF} \circ \otimes_\mathbb R^2 \hat\Phi_\ast$ on $\mathcal L_{\bar 0}^{\prime \mathcal N}$. According to the proof of Proposition \ref{nabla} it remains to show: 
\begin{align}\label{eqeq} 
p^\mathcal N\circ\Phi_\ast\circ p^\mathcal M\circ \nabla^\mathcal M= p^\mathcal N\circ \nabla^\mathcal N\circ \otimes^2(p^\mathcal N\circ \Phi_\ast )
\end{align}
In local coordinates $\Phi_\ast $ and $p^\mathcal N\circ \Phi_\ast$ only differ in terms increasing the $\mathbb Z$-degree by 2 or more. On the right hand side of (\ref{eqeq}) these terms are finally erased by $p^\mathcal N$.  Hence it equals $p^\mathcal N\circ\Phi_\ast\circ \nabla^\mathcal M$. With analogue arguments it is possible to insert $p^\mathcal M$.
\end{proof}
Two different identifications of a supermanifold $\mathcal M$ with its Batchelor model can be regarded as a diffeomorphism of supermanifolds with connections on the model $(M,\Gamma_{\Lambda E^\ast})$ onto itself with underlying map $id_M$. Hence we can finally follow:
\begin{cor}
For a supermanifold $\mathcal M$ with connection $\nabla$ the associated manifold with connection $(E,\nabla^{TE})$ is well defined up to diffeomorphisms of manifolds with connection induced by smooth bundle automorphisms of $\pi:E\to M$.
\end{cor}

\section{Geodesics}
In this section we give the identifications of geodesics of $(\mathcal M,\nabla)$ and those of $(E,\nabla^{TE})$. 
Let $\Phi:\sbb \to \cm$ be a supercurve and $\nabla$ be a connection on $\cm$. It is shown in \cite{SG-TW} that $\nabla$ can be pulled back in a unique way to an $\bb$-linear map
$$ \hat \nabla: \Phi_\ast \mathcal T_{\sbb} \otimes_\mathbb R Der(\mathcal O_\mathcal M ,\Phi_\ast \mathcal O_{\sbb})\to  Der(\mathcal O_\mathcal M ,\Phi_\ast \mathcal O_{\sbb})$$
being $\Phi_\ast \mathcal O_{\sbb}$-linear in the first argument and satisfying the appropriate  Leibniz rule in the second argument. Filling in the vector field $\Phi_\ast {\partial_t} \in \Phi_\ast\mathcal T_{\sbb}$ in the first argument we obtain the even $\bb$-linear operator $\frac{\nabla}{d t}$ satisfying $$\frac{\nabla}{d t}(fX)= \frac{\partial f}{\partial t} \cdot  X+(-1)^{|X||f|}f \cdot \frac{\nabla}{d t}(X)$$ for $f \in \Phi_\ast \mathcal O_{\sbb}$ and $X \in Der(\mathcal O_\mathcal M ,\Phi_\ast \mathcal O_{\sbb})$.
\begin{de} Let $\cm$ be a supermanifold equipped with a connection $\nabla$. A morphism 
$\Phi:\sbb\to \cm$ is called a geodesic with respect to $\nabla$ if and only if 
\begin{equation}
 \frac{\nabla}{dt}(\Phi_*{\partial_t})=0.\label{eqGeo}
\end{equation}
\end{de}
\begin{rem}
(1) Note that besides this definition of a geodesic as it appears in \cite{SG-TW}, there exists a different definition analyzed by Goertsches (see \cite{OG}) regarding more than one parameter as dynamic parameters. Details on the difference can be found in \cite{SG-TW}. \\
(2) The above definition and the subsequently quoted proposition are given in \cite{SG-TW} for the more general notion of curves $\Phi:\bb \times \mathcal S \to \mathcal M$ with arbitrary supermanifold $\cs$.
\end{rem}
We quote from \cite{SG-TW}:
\begin{prop} \cite{SG-TW}
 Let $(\cm,\nabla)$ be a supermanifold with connection. For any initial condition in $T\cm(\mathbb R^{0|1})$ there exists a unique geodesic $\Phi:\sbb\to \cm$. In local homogeneous coordinates $(q_s)$ on $\mathcal M$ a geodesic satisfies, and is well defined up to the initial condition, by the system of differential equations in $\mathcal O_{\mathbb R^{1|1}}$:
 \begin{eqnarray}
  \pa[t]^2 \Phi^*(q_s)  + \sum_{u,r}  \pa[t]\Phi^*(q_u)\pa[t]\Phi^*(q_r)\Phi^*(\Gamma_{ur}^s)=0\quad\forall\ s.
\label{equageodesic}
\end{eqnarray}
\end{prop}

We identify $\cmt$ with a Batchelor model $\mathcal O_\mathcal M=\Gamma_{\Lambda E^\ast}$ as before. A supercurve $\Phi:\sbb\to\cm$ is then well defined by the even morphism of sheaves of $\mathcal C^\infty_M$-modules $\Phi^\ast|_\mathcal F:\mathcal F \to \Phi_\ast \mathcal O_{\mathbb R^{1|1}}$. In local coordinates as given in the introduction, the equations (\ref{equageodesic}) become due to the low odd dimension of $\sbb$:
\begin{align}\label{sys}\nonumber
 &\pa[t]^2 \Phi^*(x_i)  + \displaystyle{\sum_{j,k}}  \pa[t]\Phi^*(x_j)\pa[t]\Phi^*(x_k)\Phi^*(\Gamma_{jk}^i)=0\\ 
 &\pa[t]^2 \Phi^*(e_\alpha^\ast)+\displaystyle{\sum_{j,k}} \pa[t]\Phi^*(x_j)\pa[t]\Phi^*(x_k)\Phi^*(\Gamma_{jk}^\alpha) \\ & \qquad \qquad\qquad\qquad + \displaystyle{\sum_{j,\beta}} \pa[t]\Phi^*(x_j)\pa[t]\Phi^*(e_\beta^\ast)\left(\Phi^*(\Gamma_{j\beta}^\alpha)+\Phi^*(\Gamma_{\beta j}^\alpha)\right)=0\nonumber 
\end{align}
for all $i$ and $\alpha$. Note further that by Lemma \ref{chris} and again due to the low odd dimension of $\sbb$, we can replace in (\ref{sys})
\begin{align*}
&\Phi^*(\Gamma_{jk}^i) \quad \mbox{ by  }\quad (\Phi^*\circ \psi^{-1})({\Gamma^{TE}}_{jk}^i) \ , \qquad&&\Phi^*(\Gamma_{j\beta}^\alpha) \quad \mbox{ by  }\quad (\Phi^*\circ \psi^{-1})({\Gamma^{TE}}_{j\beta}^\alpha) \ ,\\ &\Phi^*(\Gamma_{\beta j}^\alpha) \quad \mbox{ by  }\quad (\Phi^*\circ \psi^{-1})({\Gamma^{TE}}_{\beta j}^\alpha) \qquad \mbox{ and }\qquad &&\Phi^*(\Gamma_{jk}^\alpha) \quad \mbox{ by  }\quad (\Phi^*\circ \psi^{-1})({\Gamma^{TE}}_{jk}^\alpha)
\end{align*}
without modifying the differential equations for the $\Phi^\ast(x_i)$ and $\Phi^\ast(e_\alpha^\ast)$.\par \smallskip
As it is mentioned in section \ref{curve}, $\Phi$ can be associated to the curve 
$\gamma_\Phi:\mathbb R \to E$ which is well defined by the morphism of sheaves of $\mathcal C^\infty_M$-modules  $\gamma_\Phi^\ast|_{\mathcal F^\prime}:\mathcal F^\prime \to \gamma_{\Phi\ast}\mathcal C^\infty_\mathbb R$. The identification of curves $\Phi$ and $\gamma_\Phi$ is given on $\mathcal F$ by $T \circ \Phi^\ast|_\mathcal F=\gamma_{\Phi}^\ast|_{\mathcal F^\prime} \circ \psi$ with $T(f+h\tau ):=f+h$, with $\tau$ being the odd parameter in $\mathcal O_{\sbb}$. Replacing this information in (\ref{sys}), we obtain by Lemma \ref{chris} the classical defining equations in $\gamma_{\Phi}^\ast(x_i)$ and $\gamma_{\Phi}^\ast(e_\alpha^\ast)$ for geodesics of $(E,\nabla^{TE})$. An initial condition $A\in T\cm(\mathbb R^{0|1})$ reduces to a pair $\alpha_A=(\tilde A,\hat A)$ consisting of the underlying point $\tilde A\in TM$ and the image $\hat A$ of $\partial_\tau$ in $Der_{\tilde A}(\mathcal O_{T\cm},\mathbb R)_{\bar 1}\cong TE_{\tilde A}$. This gives the correspondence of initial conditions in section \ref{ini}. Together we obtain the second part of Theorem \ref{start}. More precisely we have:
\begin{cor}
 The  bijective identification of initial conditions for geodesics in $(\mathcal M,\nabla)$, resp. in $(E,\nabla^{TE})$ given in Corollary \ref{cor:xy}, induces a bijective identification of the geodesics on both objects. This identification  is the restriction of the identification of curves in Corollary \ref{T}.
\end{cor}

\section*{Appendix A}
A connection on a supermanifold $\cm\cong \Pi E$ induces a decomposition of $\mathcal T_\mathcal M$. Here we define reductions of a connection $\nabla$ and a Riemannian metric $g$ on $\cm$ associated to this decomposition. In the case $|g|=1$ the reduction yields an isomorphism $E \cong TM$ as it is stated in Remark \ref{rem:b}. In any case  the reduction carries exactly the information $\widetilde{g}_{rs} \circ \pi$ for $r,s$ even or odd indexes mentioned in Remark \ref{rem:a}.\par\smallskip
Let $\beta:M\to \mathcal M$ be the reduction map and let $g$ be an even or odd Riemannian metric on $\cm$. 
Any connection $\nabla$ on $\mathcal M$ -- notably the Levi-Civita connection on $(\cm,g)$ -- induces a connection $\nabla^E: \mathcal T_M \otimes_\mathbb R \Gamma_{E} \to \Gamma_{E}$ in the following way: for a section $s\in \Gamma_E(U)$ over $U$ open in $M$  let $\partial_s \in \mathcal T_\mathcal M(U)$ be the derivation given by contraction on $\Gamma_{\Lambda E^\ast}(U)$ with $s(x)$ at $x\in U$ and set
$$ \nabla^E(X, s):=(p\circ\nabla)(\beta_\ast(X),\partial_s)\ $$
for $X \in \mathcal T_M(U)$ and $p$ as in section \ref{sec:con}. Note that $\beta_\ast(X)$ is only well defined in $\mathcal T_\mathcal M$ up to an error term of degree at least two in the Grassmann algebra. This term is canceled by $p$. Since $(p\circ\nabla)(\beta_\ast(X),\partial_s) \in \mathcal L_{\bar 1}$ it makes sense to regard $\nabla^E(X, s)$ as an element in $\Gamma_E$ again. 
We quote a useful tool from \cite{Monterde}: 
\begin{theo}\cite{Monterde} \label{teil}
 A connection $\nabla^{E}: \mathcal T_M \otimes_\mathbb R \Gamma_{E} \to \Gamma_{E}$ on the vector bundle ${E} \to M$ induces on the supermanifold $\mathcal M=\Pi E^\ast$ an isomorphism 
 $$ \mathcal T_\mathcal M  \cong \Gamma_{\Lambda E^\ast} \otimes_{\mathcal C^\infty_M} (\mathcal T_M \oplus \Gamma_E)\ .$$
\end{theo}
So let $(\cm, g)$ be a Riemannian supermanifold with $g:\mathcal T_\mathcal M\otimes\mathcal T_\mathcal M \to \mathcal O_\mathcal M$ and let $\nabla^E$ be the connection on the Batchelor model $E \to M$ induced by the Levi-Civita connection of $g$. We obtain from the composition $\beta^\ast \circ g$ via the isomorphism of the theorem:
\begin{itemize}
 \item in the case $|g|=0$ a pseudo-Riemannian metric $g^{TM}$ on $M$ and a non-degenerate 2-form $\omega^E:\Gamma_E \otimes_{\mathcal C^\infty_M}\Gamma_E \to \mathcal C^\infty_M$,
 \item in the case $|g|=1$ a bundle isomorphism $B^E: E \to TM$   
\end{itemize}
Note that in both situations there is a canonical converse: starting with an arbitrary choice of $(g^{TM},\omega^E, \nabla^E)$, resp. $(B^E, \nabla^E)$ we can reconstruct via Theorem \ref{teil} a Riemannian metric $g$ on $\mathcal M$. But since we loose information in the reduction we will not get back the original metric on $\cm$ in general.

\end{document}